\documentclass[11pt]{amsart}

\usepackage[T1]{fontenc}		
\usepackage[utf8]{inputenc}

\usepackage{amsmath, amscd,amssymb}
\usepackage{graphicx,psfrag,epsfig}		
\usepackage[all]{xy}
\usepackage{enumerate}

\begin{document}

\title{Brian\c con-Speder examples and  the failure of weak Whitney regularity}
\author{Karim Bekka and David Trotman} \thanks{K. B. :  IRMAR, 
Universit\'e de Rennes 1, Campus Beaulieu, 35042 Rennes, France ; 
D. T. : LATP, UMR 7353, Aix-Marseille Universit{\'e}, 39 Rue Joliot-Curie, 13453 Marseille, France.}
\email{karim.bekka@univ-rennes1.fr\\ David.Trotman@cmi.univ-mrs.fr}
\date{September 27, 2012}
\markboth{Regular stratification}{K. Bekka}       
 \maketitle
\theoremstyle{plain}
\newtheorem{Thm}{Theorem} 
\newtheorem{Def}{Definition}
\newtheorem{Lm}{Lemma}                             
\newtheorem{Rem}{Remark}

\def\bdef{\begin{Def}}
\def\endef{\end{Def}}
\def\bthm{\begin{Thm}}
\def\ethm{\end{Thm}} 
\def\blm{\begin{Lm}}
\def\elm{\end{Lm}}
\def\brm{\begin{Rem}}
\def\erm{\end{Rem}}
\def\beq{\begin{eqnarray}}
\def\eneq{\end{eqnarray}}
\def\Cal{\mathcal}
\def\Bbb{\mathbb}
\def\gm{\gamma}
\def\dt{\delta}
\def\eps{\epsilon}
\def\lb{\lambda}
\def\Bbb{\mathbb}
\def\R{\Bbb R}
\def\P{\Cal P}
\def\C{\Bbb C}
\def\Z{\Bbb Z}

\section{Introduction} 
 
 In \cite{3, 5} we introduced a weakened form of Whitney's condition $(b)$, motivated by the work of M. Ferrarotti on metric properties of Whitney stratified sets \cite{12, 13}. The resulting weakly Whitney stratified sets retain many properties of Whitney stratified sets. They are Bekka $(c)$-regular, as recalled in section 5 below. It follows, by the thesis of the first author \cite{3, 4}, that they have the structure of abstract stratified sets \cite{19}, and thus are locally topologically trivial along strata \cite{27, 19}, and are triangulable \cite{17}. Weakly Whitney stratified sets  also have many of the metric properties known to hold for  Whitney stratified sets, as we proved  in  \cite{7}.  Ferrarotti  {\cite{14, 15}, Orro and Trotman \cite{22}, Parusinski  {\cite{24}, Pflaum  {\cite{25}, and Sch\"urmann  \cite{26} have described and developed further useful properties of weakly Whitney stratified sets.

It is easy to find real algebraic varieties with weakly Whitney regular stratifications which are not Whitney regular, and we give such an example in section 3 below. No examples are known among complex analytic varieties however, so that the natural question arises : do Whitney regularity and weak Whitney regularity coincide in the complex case ? As a test,  in this paper we study the well-known Brian\c con-Speder examples, consisting of Milnor number constant families of complex surface singularities in ${\bf C}^3$ which are not Whitney regular  \cite{10}, although they are $(c)$-regular, to determine if they are weakly Whitney regular.

We investigate systematically all of the (infinitely many) Brian\c con-Speder examples, and establish in particular that none of these examples are weakly Whitney regular. We determine {\bf all} the complex curves along which Whitney $(b)$-regularity fails and {\bf all} the complex curves along which weak Whitney regularity fails. It turns out that for each example there are a finite number of curves $\gamma_i$ such that weak Whitney regularity fails precisely along those curves tangent to one of the $\gamma_i$ at the origin. For example, the classical Brian\c con-Speder example $f_t(x,y,z) = x^5 + txy^6 + y^7z + z^{15}$ for which  $\mu (f_t) = 364$, has $16$ such curves $\gamma_1, \dots, \gamma_{16}$ defining thus all the curves on which weak Whitney fails, where each $\gamma_i (s)$ is of the form $(s^8, a s^5, 4 a^{-7} s^5,  - 5 a^{-6} s^2) \in {\bf C}^4$, with $a^{16} = -8$ (hence the $16$ distinct complex solutions).

 It should be of interest to interpret these curves in the light of other studies of the metric geometry of singular complex surfaces, for example the recent work of Birbrair, Neumann and Pichon characterising their inner bilipschitz geometry \cite{8}, and the same authors' work characterising outer bilipschitz triviality \cite{9}, or the work of Garcia Barroso and Teissier on the local concentrations of curvature \cite{16}.

Further evidence that weak Whitney regularity and Whitney regularity might be equivalent for complex analytic stratifications, at least for complex analytic hypersurfaces,  comes from the recent result of the second author that equimultiplicity of a family of complex analytic hypersurfaces follows from weak Whitney regularity of the stratification. 

The second author acknowledges the support of the University of Rennes 1 during several visits to Rennes, when much of the work in this paper was done.

\section{Definitions.}

We start by recalling the Whitney conditions.

\medskip

Let $X, Y$ be two submanifolds of a riemannian manifold $M$ and take $y \in  X \cap Y$.
\medskip

{\bf Condition $(a)$}:
The triple $(X, Y, y)$ satisfies Whitney's condition $(a)$ if for each sequence of points $\{ x_i \} $ of $X$ converging to $y \in  Y$ such that $T_{x_i}X$
converges to $\tau$ (in the corresponding grassmannian in $TM$), then $T_yY \subset \tau$.

\medskip

{\bf Condition $(b)$}:
The triple $(X, Y, y)$ satisfies Whitney's condition $(b)$ if for each local diffeomorphism $h : \mathbb{R}^n \rightarrow M$ onto a neighbourhood $U$ of $y$
in $M$ and for each sequence of points $\{ (x_i, y_i) \}$ of $h^{-1}(X) \times h^{-1}(Y)$ converging
to $(h^{-1}(y), h^{-1}(y))$, such that the sequence $\{T_{x_i}h^{-1}(X)\}$ converges to $\tau$ in the
corresponding grassmannian and the sequence $\{ \overline{x_iy_i} \}$ converges to $\ell$ in $  \mathbb{P}	^{n-1}({ \mathbb{R}	 })$, then $\ell \subset \tau$.  

\medskip

{\bf Condition $(b^{\pi})$}:
The triple $(X, Y, y)$ satisfies Whitney's condition $(b^{\pi})$ if for each local diffeomorphism $h : \mathbb{R}^n \rightarrow M$ onto a neighbourhood $U$ of $y$
in $M$ and for each sequence of points $\{ x_i \}$ of $h^{-1}(X)$ converging
to $h^{-1}(y)$, such that the sequence $\{T_{x_i}h^{-1}(X)\}$ converges to $\tau$ in the
corresponding grassmannian and the sequence $\{\overline{ x_i \pi(x_i )} \}$ converges to $\ell$ in $ \mathbb{P}^{n-1}({\mathbb{R}})$, then $\ell \subset \tau$.  

\medskip

 One says that $(X, Y)$ satisfies condition $(a)$ (resp.$(b)$, $(b^{\pi})$) if $(X, Y, y)$ satisfies $(a)$
(resp. $(b)$, $(b^{\pi})$) at each $y \in X \cap Y $.

\medskip

\noindent {\bf Remark 2.1.} It is an easy exercise to check that condition $(b)$ implies condition $(a)$ \cite{19}. Also $(b)$ is equivalent to both $(a)$ and $(b^{\pi})$ holding \cite{20}.

\medskip

We now introduce a regularity condition $(\delta)$, obtained by weakening condition $(b)$.

Given a euclidean vector space $V$, and two vectors $v_1, v_2 \in V^*= V-\{0\}$, define the sine
of the angle $\theta (v_1, v_2)$ between them by :
$$
\sin \theta (v_1, v_2) = \frac {||v_1\wedge v_2||}{||v_1||.||v_2||}
$$

\noindent where $v_1\wedge v_2$ is the usual vector product  and
$||.||$ is the norm on $V$ induced by the euclidean structure.
Given two  vector subspaces  $S$ and $T$ of $V$ we define the sine of the angle between
$S$ and $T$ by :
$$
\sin \theta (S,T) =\sup \{ \sin\theta (s,T) : s\in S^*\}
$$
where 
$$
\sin \theta (s,T) =\sup \{ \sin\theta (s,t) : t\in T^*\}.
$$
 If
$\pi_T : V\longrightarrow T^\perp $ is the orthogonal projection onto the
 orthogonal complement of $T$, then
 $\sin \theta (s,T) =\frac {||\pi_T(s)||}{||s||}$. 
 The definition for lines
 is similar to that for
vectors - take unit vectors on the lines.

 One verifies easily that :
$
 \,\sin \theta (v_1, v_3) \leq \sin \theta (v_1, v_2) + \sin \theta (v_2,
v_3)$ for all $ v_1, v_2, v_3 \in V^*,$
and
$ \sin \theta(S_1+S_2, T) \leq \sin \theta(S_1, T)+\sin \theta(S_2, T),
$
for subspaces $S_1, S_2$, $T$ of $V$ such that $S_1$ is
orthogonal to $S_2.$

\medskip

{\bf Condition $(\delta)$}:
We say that the triple $(X,Y,y)$ satisfies condition ($\delta$) if 
 there exists a local diffeomorphism $h : \mathbb {R}^n \longrightarrow  M
$ to a neighbourhood $U$ of $y$ in $M$, and there exists a real number
$\delta_y$,
$0 \leq \delta_y < 1$, such that for every sequence  $\{ x_i, y_i\}$
of $h^{-1}(X)\times h^{-1}(Y)$ which converges to ($h^{-1}(y), h^{-1}(y)$)
such that the sequence $\overline{x_iy_i}$ converges to $l$ in
$\mathbb{P}^{n-1}(\mathbb {R})$,
and the sequence $T_{x_i}h^{-1}(X)$ converges to $\tau$, then
$\sin \theta(l, \tau ) \leq \delta_y$. 
\bigskip

\noindent {\bf Remark 2.2.}
Clearly condition ($b$) implies 
$(\delta)$ : just take $\delta_y = 0$.

\medskip

{\bf Definition.}
A {\it weakly Whitney stratification}  of a subspace
 $A$ of a manifold $M$ is a locally finite partition 
of $A$ into connected submanifolds, called the {\it strata}, such that :

1 ) - Frontier Condition : If $X$ and $Y$ are distinct strata
such that $\overline X \cap Y \neq \emptyset $, that is
 $X$ and $Y$ are {\it adjacent},
then $Y\subset \overline
X$.

2 ) - Each pair of adjacent strata satisfies condition $(a)$. 

3 ) - Each pair of adjacent  strata satisfies condition $(\delta)$.

\medskip

{\bf Examples.}

\medskip
1. Every Whitney stratification is weakly Whitney regular.

\medskip
2. 
Let $X$ be the open logarithmic spiral with polar equation,

$$\{ (r,\theta)\in\mathbb {R}^2 \, |\, r = e^{\frac {t}{\tan (\beta) }},
 \theta = t (\textrm{mod}\, 2\pi)\} 
\quad\textrm{ where}\quad  0 <\beta  < \frac\pi{2} \}$$

\noindent and let  $Y=\{0\} \subset \mathbb {R}^2$. Condition $(a)$ is trivially
satisfied for $(X,Y,\{0\})$, and condition $(\delta )$ is also
satisfied, but condition 
$(b)$ fails because the angle $\theta(\overline{x0}, T_x X) = \beta$
 is constant and nonzero for all $x$ in $X$. So this is a weakly Whitney regular stratification which is not Whitney regular.
 
\medskip
3. If $X$  is the open spiral with  polar equation
$$\{ (r,t)\in\mathbb {R}^2\, |\, r = e^{-\sqrt {t}}, t\geq 0  \}$$
 and $Y=\{0\} \subset \mathbb {R}^2$, then
the stratified space $X \cup Y$ is not weakly Whitney.

\medskip

\noindent {\bf Remark 2.3.} 
In the definition of weakly Whitney stratification, we could further 
weaken condition $(\delta )$ as follows :
If $\pi $ is a local $C^1$ retraction associated to a $C^1$
tubular neighbourhood of $Y$ near $y$,
 a condition $(\delta^\pi )$ is obtained from the definition of $(\delta)$ 
by replacing the sequence
$\{y_i\}$ by the sequence $\{ \pi (x_i)\}$.
Clearly $(b^{\pi}$ implies $(\delta^{\pi})$. Recall that $(b) \Longleftrightarrow (b^{\pi}) + (a)$ {\cite{NT}}, as noted above.

\medskip

\begin{Lm} {\it $(\delta ) + (a) \Longleftrightarrow (\delta^\pi ) + (a). $}

\end{Lm}

\noindent {\bf Proof.} Clearly $(\delta ) \Longrightarrow (\delta^\pi )$,
so it suffices to show that $(\delta^\pi ) + (a) \Longrightarrow (\delta ) $.
In the definition of $(\delta )$ decompose the limiting vector
 $l$ as the sum of a vector $l_1$
tangent to $Y$ at $y$, and a vector $l_2$ tangent to $\pi^{-1}(y)$
at $y$.
Then $\sin \theta(l , \tau ) = \sin \theta(l_1 + l_2 , \tau ) \leq \sin
\theta(l_1, \tau ) +  \sin \theta(l_2, \tau ).$
By condition $(a)$, $\sin \theta(l_1, \tau ) = 0 $, hence
$\sin \theta(l , \tau ) \leq   \sin \theta(l_2, \tau )$ which is less than or equal to $ \delta_y $
by hypothesis, 
implying  $(\delta )$. \qed

\medskip
This will make checking weak Whitney regularity easier.

\section{ Real algebraic examples.}

Because many of the important applications of Whitney stratifications arise in algebraic geometry, it is necessary to know how weak Whitney regularity compares with Whitney regularity for semi-algebraic or real algebraic stratifications, as well as for complex algebraic/analytic stratifications. The following simple example illustrates that weak Whitney regularity is strictly weaker than Whitney regularity for real algebraic stratifications. No such example is currently known in the case of complex algebraic stratifications, and this will be the motivation for the calculations in sections 7, 8 and 9 of this paper.

\medskip

\noindent ${\textrm{\bf Example I.}}$

Let $V=\{(x,y,t)\in\mathbb{R}^3\,|\, y^{6} = t^6x^2+x^6 \}$, let
 $Y$ denote the $t$-axis, and let $X = V \setminus Y$.
One can check that the triple $(X,Y,(0,0,0))$ satisfies conditions $(a)$ and
($\delta$), but not condition $(b)$. 
See {\cite{BT2}} for details.
\bigskip

The following example illustrates the independence of the conditions $(a)$ and $(\delta)$ in the case of real algebraic stratifications..

\medskip

\noindent${\textrm{\bf Example II.}}$

Let $V=\{(x,y,t)\in\mathbb{R}^3\,|\,  y^{20} = t^4x^6+x^{10} \}$, 
let $Y$ denote the $t$-axis and let $X = V \setminus Y$.
Then the triple $(X,Y,(0,0,0))$ satisfies condition $(\delta)$, but not condition $(a)$. 
 For details see \cite{6}.

\section { Some properties of weakly Whitney stratified spaces.}

Like Whitney stratified spaces, weakly Whitney stratified spaces are filtered by dimension.

\medskip
 
{\bf Proposition 4.1 \cite{6}.}
{\it Suppose that a triple $(X, Y, y)$, $y \in Y \cap \overline{X}$,  satisfies conditions 
$(a)$ and $(\delta)$. 
Then $dim\, Y < dim\,X$.}
\medskip

{\bf Definition.} If $(A, \Sigma)$, $(B, \Sigma ')$
are weakly Whitney stratified spaces in $M$, then $(A, \Sigma )$ and
  $(B, \Sigma ')$
 are said to be {\it in general position} if for each pair of strata $X\in
\Sigma$ and $X'\in \Sigma '$, $X$ and $X'$ are in general position in 
$M$, i.e. the natural map :

$$
T_xM \longrightarrow T_x M/T_x X \oplus  T_x M/T_x X'
$$
 is surjective for all $x\in X\cap X'$.

\medskip

{\bf  Proposition 4.2.}
 {\it Let $V$ be a submanifold of $M$ in general position with respect to
 $ (A, \Sigma )$. Then  $(A\cap V, \Sigma \cap V)$ is weakly Whitney regular, if $(A,\Sigma)$ is weakly Whitney regular.}
   
\medskip

A proof is given in \cite{6}. A stronger statement, in the case of  two stratified sets transverse to each other, is given in \cite{23}.

\medskip

If $A$ is locally closed and
$(A, \Sigma )$ is weakly Whitney  
(without assuming the frontier condition) then the stratified space
$(A, \Sigma_c )$, whose strata are the connected components
of the strata of $\Sigma $, automatically satisfies the frontier condition. 
See \cite{3, 4} for the $(c)$-regular case, which includes
the case of weakly Whitney stratifications, as remarked below.

\medskip

{\bf Proposition 4.3.}
{\it Let $f : M \to M'$ be a $C^1$ map, and let $(A, \Sigma)$ be a weakly Whitney stratified
 space in $M'$. If $f$ is transverse to each stratum $X\in \Sigma$, then the pull-back 
$(f^{-1}(A), f^{-1}(\Sigma ))$ is  weakly Whitney stratified.}

\medskip

See \cite{6} for proofs.

\section {  $(c)$-regularity of weakly Whitney stratifications.}

In this section we recall the fact that weakly Whitney stratified spaces
are $(c)$-regular.
 It follows \cite{3, 4} that they can be given 
 the structure of abstract stratified sets in the sense of Thom-Mather \cite{19}, 
 implying in particular local topological 
triviality along strata and triangulability \cite{17}.

\medskip

Let $(U, \phi)$ be a $C^1$ chart at $y$ for a submanifold
$Y \subseteq M$ where
  $dim Y = d$,
$$\phi : (U, U\cap Y, y) \longrightarrow (\mathbb{ R}^n, \mathbb{ R}^d \times  \{0\}^{n-d}, 0).$$
Then  $\phi$  defines a tubular neighbourhood $T_\phi $ of $U\cap Y$ in $U$,
induced by the standard tubular neighbourhood of $\mathbb{R}^d\times  \{0\}^{n-d}$ in
$\mathbb{ R}^n$ :

- with retraction 
$\pi_\phi = \phi^{-1} \circ \pi_d \circ \phi $ where $ \pi_d :
\mathbb{R}^n\to \mathbb{ R}^d$ is 
 the canonical projection, 

- and distance function $\rho_\phi = \rho \circ \phi : U \to \mathbb {R^+}$ where
$\rho : \mathbb{ R}^n \to \mathbb{ R}^+$ is the function defined by 
$\rho (x_1, \cdots , x_n) = \Sigma_{i=d+1}^n x_i^2$.
\smallskip

It is well-known (see \cite{19, 28, 29}) that if a pair $(X, Y)$ of submanifolds of $M$ 
satisfies Whitney's condition $(b)$ then for {\it any} sufficiently small 
tubular neighbourhood
$T_Y$ of $Y$ in $M$, the map
$$
(\pi_Y, \rho_Y)|_{X\cap T_Y} : X\cap T_Y\longrightarrow Y\times \mathbb{ R} 
$$
is a submersion. In fact this property characterises $(b)$-regularity \cite{29}.
 For comparison, when the pair $(X, Y)$ is 
 weakly Whitney, there exists {\it some} tubular neighbourhood $T_Y$
such that the map
$$
(\pi_Y, \rho_Y)|_{X\cap T_Y} : X\cap T_Y\longrightarrow Y\times \mathbb{ R} 
$$
is a submersion. 

\medskip

{\bf Proposition 5.1.}
 {\it  Let $X, Y$ be two submanifolds of $M$, $Y\subset \overline
X$ and let $y\in Y$.
  If the triple $(X, Y, y)$ satisfies the weak Whitney conditions,
then there exists a chart $C^1,$ $(U, \phi )$ 
at  $y$ for $Y$ in $M$ (the ambient manifold) and a neighbourhood
$U'$ of $y$, $U'\subset U$ such that $ (\pi_\phi, \rho_\phi )|_{U'\cap X}$ 
is a submersion.}

\medskip 

{\bf Corollary.} 
{\it Let $X, Y$  be two submanifolds of $M$ such that $Y\subset
\overline X$ and the pair  $(X, Y)$ satisfies the
conditions $(a)$ and $(\delta )$. Then there exists a tubular neighbourhood
 $T_Y$ of
$Y$ in $M$ such that $(\pi_Y , \rho_Y )|_X : 
 X\cap T_Y \longrightarrow Y\times \mathbb{ R}
$ is a submersion.}

\medskip

{\bf Proposition 5.2.}
{\it Every weakly Whitney stratified space is $(c)$-regular.}

\medskip

For the proofs see \cite{6}. We note that, when weak Whitney regularity holds, the control function in the definition of $(c)$-regularity can be chosen to be a standard distance function arising from a tubular neighbourhood. This means  that weak Whitney regularity is a much stronger condition than mere $(c)$-regularity, for which the control function may be weighted homogeneous or even infinitely flat along $Y$.

\section  {Complex stratifications.}

\medskip

In Example I we saw an  example of a weakly Whitney regular real algebraic stratification in $\R^3$ which is not Whitney $(b)$-regular. 
We are now interested in comparing weak Whitney regularity and Whitney regularity of {\it complex} analytic or {\it complex} algebraic stratifications, the main question being whether the extra `rigidity' of complex analytic varieties prevents the existence of weakly Whitney complex analytic stratifications which are not Whitney regular.

\medskip

Let $F$ be an analytic function germ from $\C^n\times \C$ to $\C$, defined in a neighbourhood of $0$,
$$\begin{matrix} F: & \C^n\times \C,0 &\longrightarrow& \C,0 \\ & (x,t)&\longmapsto & F(x,t)  \end{matrix}$$
where $F(0,t)=0.$
We denote by $\pi$ the projection on the second factor, and let $V=F^{-1}(0)$, $Y=\{0\}^n \times \C$ and $V_{t}=\{x\in \C^n\,|\;F(x,t)=0\}.$
We assume that each $V_{t}$ has an isolated singularity at $(0,t)$, the critical set of the restriction of $\pi$ to $V$ is $Y,$ and $X=V\setminus Y$ is an analytic complex manifold of dimension $n$.

For each point $(x,t)\in X$ we have 
$$T_{(x,t)}X=\left\{(u,v)\in \C^n\times\C\, \left| \sum_{i=1}^{n}u_{i}\frac{\partial F}{\partial x_{i}}(x,t) +v\frac{\partial F}{\partial t}(x,t)=0\right.\right\}=\left(\C \overline{grad} F\right)^{\perp}.$$
Let $grad F=(\frac{\partial F}{\partial x_{1}},\ldots, \frac{\partial F}{\partial x_{n}},\frac{\partial F}{\partial t})$, $grad_{x} F=(\frac{\partial F}{\partial x_{1}},\ldots, \frac{\partial F}{\partial x_{n}})$
and $\| grad_{x} F\|^2=  \sum_{i=1}^{n}\|\frac{\partial F}{\partial x_{i}} \|^2.$

\medskip

The following characterisations of conditions $(a)$, $(b^{\pi})$ and $(\delta^{\pi})$ are straightforward.

\medskip

\underline { Whitney's  condition $(a)$ }

The pair $(X,Y)$ satisfies Whitney's condition $(a)$ at $0$  if and only if 

$$\lim_{\substack{(x,t)\to 0\\ (x,t)\in X}}\left ( \frac{\frac{\partial F}{\partial t}(x,t)}{\| grad_{x} F(x,t)\|}\right)=0.$$

\underline { Whitney's  condition $(b^{\pi})$ }

The couple $(X,Y)$ satisfies Whitney's condition $(b^\pi)$ at $0$ if and only if 

$$\lim_{\substack{(x,t)\to 0\\ (x,t)\in X}}\left (\frac{\sum_{i=1}^{n}x_{i}\frac{\partial F}{\partial x_{i}}(x,t)}{\|x\|\| grad_{x} F(x,t)\|}\right)=0.$$

\underline {  Condition $(\delta^{\pi})$ }

The pair $(X,Y)$ satisfies the $(\delta^{\pi})$ condition  at $0$ if and only if  
there exists a real number $0\leq\delta<1$ such that 

$$\lim_{\substack{(x,t)\to 0\\ (x,t)\in X}}\left (\frac{\sum_{i=1}^{n}x_{i}\frac{\partial F}{\partial x_{i}}(x,t)}{\|x\|\| grad_{x} F(x,t)\|}\right)\leq\delta.$$

Recall that Whitney's condition $(b)$ implies  $(a)+(\delta^{\pi}).$
  
  \medskip
 {\bf Question.} Is the converse true in the complex hypersurface case,  i.e.  does $(a)+(\delta^{\pi})$ imply $(b)$ or, equivalently, does $(a)+(\delta^{\pi})$ imply $(b^{\pi})$ ?

    \medskip
       
\noindent {\bf Remark 6.1.} We know by the fundamental result of L\^e Dung Tr\`ang and K. Saito {\cite{18}} that a family of complex hypersurfaces with isolated singularities has constant  Milnor number if and only if 

$$\lim_{(x,t)\to 0}\left ( \frac{\frac{\partial F}{\partial t}(x,t)}{\| grad_{x} F(x,t)\|}\right)=0,$$

\noindent which implies  condition $(a)$.

The following lemma due to Briançon and Speder {\cite{11}} gives an equivalent condition to $(b^{\pi})$ when $(a)$ is satisfied.

Let $\gamma:([0,1],0)\to ( \mathbb{C}^n\times \mathbb{C},0)$, be a germ of an analytic arc and $\nu$ the valuation along $\gamma$
in the local ring $\mathcal{O}_{n+1,0}.$ 

{\bf Notation:} $\nu(x):=\inf\{\nu(x_i)| 1\leq i\leq n\}$ and $\nu(J_x(F)):=\inf\{\nu(\frac{\partial F}{\partial x_i}| 1\leq i\leq n\}.$

\begin{Lm}\label{simple} 

The following statements are equivalent:
\begin{enumerate}
\item
the pair $(X,Y)$ satisfies $(b^\pi)$  at $0$,

\item 
$$\lim_{\substack{(x,t)\to 0\\ (x,t)\in X}}\left (\frac{t\frac{\partial F}{\partial t}(x,t)}{\|x\|\| grad_{x} F(x,t)\|}\right)=0.$$
\end{enumerate}
 
In other words,
the following statements are equivalent:
\begin{enumerate}
\item $\nu( \sum_{i=1}^{n}x_{i}\frac{\partial F}{\partial x_{i}})>\nu(x) +\nu(J_x(F))$

\item $\nu(t)+\nu( \frac{\partial F}{\partial t})>\nu(x) +\nu(J_x(F))$
\end{enumerate}

where $\nu$ is the valuation along germs of analytic arcs $\gamma:[0,1]\to X$.
\end{Lm}
\begin{proof}

For $s\in [0,1]$, $\gamma(s)=(x_1(s),\ldots,x_n(s),t(s)).$

Since $F\circ \gamma\equiv 0$, we have 

$$ \sum_{i=1}^n x'_i(s) \frac{ \partial F\circ \gamma}{\partial x_i}(s)=-t'(s) \frac{ \partial F\circ \gamma}{\partial t}(s).  \quad\quad (3)
$$

If $a=\nu(x)$ and $b=\nu(J_x(F))$, there exist two  non zero vectors of 
$\mathbb{C}^n$,  $A$ and  $B$, such that
$$\begin{cases}
 (x_1(s), \ldots,x_n(s))=As^ a+\ldots \\
(\frac{ \partial F\circ \gamma}{\partial x_1}(s),\ldots,\frac{ \partial F\circ \gamma}{\partial x_n}(s))=Bs^ b+ \ldots .
\end{cases}$$ 

We suppose  $(1)$ holds. Then since  $\sum_{i=1}^n x_i(s) \frac{ \partial F\circ \gamma}{\partial x_i}(s)=\langle A,B\rangle s^{a+b} + \ldots$ ,
we must have  $\langle A,B\rangle =0.$

From (3) we have 
$$t'(s) \frac{ \partial F\circ \gamma}{\partial t}(s) =-\sum_{i=1}^n x'_i(s) \frac{ \partial F\circ \gamma}{\partial x_i}(s)
=-a\langle A,B\rangle s^{a+b-1} + \ldots .$$

Then $\nu(t)+\nu( \frac{\partial F}{\partial t})=\nu(t'(s) \frac{ \partial F\circ \gamma}{\partial t})+1>(a+b-1) +1=a+b.$

We suppose now $(2)$ holds. Then since 
$t'(s) \frac{ \partial F\circ \gamma}{\partial t}=-a\langle A,B\rangle s^{a+b-1} + \ldots$,  we must  have again
 $\langle A,B\rangle=0$, which is exactly condition $(1)$.

\end{proof}

\section  {The Brian\c{c}on and Speder example with $\mu = 364$.}

 In this section we study the original example, due to Brian\c{c}on and Speder  {\cite{10}}, of a topologically trivial family of isolated complex hypersurface singularities which are not Whitney regular. The examples of Brian\c con and Speder given in {\cite{10}} were the {\it only} such examples known, until very recently. 

We shall carry out initially explicit calculations for the most well-known 
example of  Brian\c{c}on and Speder, analysed in their celebrated note of January 1975
 : $F(x,y,z,t) = F_t(x,y,z) =x^5+txy^6+y^7z+z^{15} $  for which $\mu (F_t) = 364$ for all $t$ near $0$.
\medskip

\underline { The Brian\c{c}on and Speder  example is not weakly Whitney  regular.  }
\medskip

Let  $F(x,y,z,t)= x^5+txy^6+y^7z+z^{15}$. Then $F$ is a quasihomogenous 
$\mu$-constant family of type $(3,2,1;15).$ Thus the 
stratification $(F^{-1}(0) \setminus (0t), (0t))$ is $(a)$-regular by Remark 6.1.

We shall construct an explicit analytic path $\gamma(s)=(x(s),y(s),z(s),t(s))$ contained in $F ^{-1}(0)$  such that the module of  $\Delta(x,y,z,t)=\left (\frac{\sum_{i=1}^{n}x_{i}\frac{\partial F}{\partial x_{i}}(x,y,z,t)}{\|x\|\| grad_{x} F(x,y,z,t)\|}\right)$ tends
to $1$ when $(x,y,z,t)$  tends to $0$ along $\gamma(s)$.  This means that condition $(\delta^{\pi})$ is not satisfied at $0$, by the characterisation given in section 6. By Lemma 1 in section 2 it then follows using $(a)$-regularity that $(\delta)$ is not satisfied at $0$, so that weak Whitney regularity fails.

Following {\cite{BS1}} and {\cite{Tr1}} we take

 $$\left\{\begin{array}{lll} x(s)=s^8 \\y(s)=a s^5\\ z(s)= {{4}\over{{a}^7}} \lambda s^5\\t(s)=-\frac{5}{a^6}s^2 \end{array}\right.$$

\noindent with $a \not =0.$

For $\gamma(s)$ to lie on $F^{-1}(0)$ we  must have that
$$ F(\gamma(s))=(1-\frac{5}{a^6}a^6+ 4 \lambda + ({{4}\over{a^7}})^{15} \lambda^{15}s^{35})s^{40}\equiv 0, $$

\noindent  so that 
$G(\lambda,s)=-4+ 4 \lambda + ({{4}\over{a^7}})^{15} \lambda^{15}s^{35} \equiv 0$.

Since $\frac{\partial G}{\partial \lambda}(\lambda,0)= 4 \not=0$,  it follows by the implicit function theorem that $\lambda $ is a function of $s$ for $s$ near $0$.

Note that $\lambda(0)=1$.

Then we have along $\gamma(s)$  near $s=0,$

$$\left\{\begin{array}{lll}
 \frac{\partial F}{\partial x}=5x^4+ty^6=5s^{32}- \frac{5}{a^6}a^6s^{32} = 0
 \\
 \frac{\partial F}{\partial y}=6txy^5+7y^{6}z =\left(\frac{-30}{a}+ {{28}\over{a}} \lambda\right)s^{35}\\
  \frac{\partial F}{\partial z}=y^7+15z^{14}=a^7s^{35} + 15 ({{4}\over{a^7}})^{14}\lambda^{14}s^{70}\sim  a^7s^{35}.
  \end{array}\right.$$

Because  $\lambda(0) = 1$, the limit of orthogonal secant vectors $\frac{(x,y,z)}{\|(x,y,z)\|}$ is  $(0: a :  {{4}\over{a^7}}) = (0:a^8:4)$,  and the limit of normal vectors $ \frac{ grad_{x} F(x,y,z,t)}{\| grad_{x} F(x,y,z,t)\|}$ is 
 $(0 : \frac{-2}{a} : a^7  ) = (0: -2 : a^8). $

Then $\Delta(\gamma(s))$ tends to $1$ if and only if 
$a^8= ({{-2}}) {{4}\over{a^8}}$.  It follows that $(\delta^{\pi})$  is not satisfied along $\gamma$ if and only if
$a^{16}=-8$.  Choosing $a$ to be one of these 16 complex 
numbers,  we have the desired conclusion, i.e. that $(\delta^{\pi})$ fails. Note however that we cannot exclude the possibility that there are other curves on which $(\delta^{\pi})$ fails. For this purpose, we shall next  make a systematic study of all curves $\gamma(s)$ on $F^{-1}(0)$ and passing through the origin.

\medskip

A similar calculation for the simpler $\mu$-constant family $F(x,y,z,t) = x^3 + txy^3 + y^4z + z^9$, for which $\mu = 56$, also due to Brian\c{c}on and Speder (see  {\cite{10}}), shows that  $(\delta^{\pi})$ fails for this example too. Our systematic study to determine all curves on which $(\delta^{\pi})$ fails for this example will be copied in a more general study given below of the infinite family of examples, of which  $x^3 + txy^3 + y^4z + z^9$ is the first, defined by Brian\c con and Speder in their celebrated 1975 note \cite{10}.

\section  {Failure of weak Whitney regularity: a complete analysis.}

Take again $F(x,y,z,t)= x^5+txy^6+y^7z+z^{15}$. 

\medskip

In what follows we determine the initial terms of {\it all} curves along which condition $(\delta)$ fails, or equivalently by Lemma 1, along which $(\delta^{\pi})$ fails.

\medskip

Let $\gamma:([0,1],0)\to ( \mathbb{C}^n\times \mathbb{C},0)$, be a germ of an analytic arc and $\nu$ the valuation along $\gamma.$

Let  $X=(x,y,z)$ and $J_XF=( \frac{\partial F}{\partial x}, \frac{\partial F}{\partial y}, \frac{\partial F}{\partial z})$,

We will use the notations $\nu(X):=\inf\{\nu(x), \nu(y),\nu(z)\}$ and $\nu(J_X(F)):=\inf\{\nu(\frac{\partial F}{\partial x}),\nu(\frac{\partial F}{\partial y}), \nu(\frac{\partial F}{\partial z})\}.$

We  begin by  determining   the curves along  which condition $(b^{\pi})$ holds (using $(a)$-regularity we know that $(b^{\pi})$ is equivalent to $(b)$ here). 

Because of  lemma \ref{simple} and the $\mu$-constant condition,  if $\nu(t)\geq \nu(X)$ then $(b')$ holds

 ($\nu(t)+\nu( \frac{\partial F}{\partial t})=\nu(t)+\nu(x)+6\nu(y)>\nu(x) +\nu(J_X(F))$)  
 
 which means that Whitney's condition $(b)$ holds when  $\nu(t)\geq\nu(X)$.\checkmark

We suppose from now on that  \fbox{$\nu(t)<\nu(X)$}.

If $\nu(\frac{\partial F}{\partial x})=\inf \{4\nu(x), \nu(t)+6\nu(y)\}$ we have:
\begin{enumerate}
\item
if $4\nu(x)\geq \nu(t)+6\nu(y),$ we must have $\nu(x)>\nu(y)$ and then  $\nu(t)+\nu( \frac{\partial F}{\partial t})=\nu(t)+\nu(x)+6\nu(y)>\nu(y)+ (\nu(t)+6\nu(y))
\geq \nu(X) +\nu(J_X(F)).$\checkmark
\item
if $4\nu(x)<\nu(t)+6\nu(y)$ then  $\nu(t)+\nu( \frac{\partial F}{\partial t})=\nu(t)+\nu(x)+6\nu(y)>\nu(x)+ 4\nu(x)
\geq \nu(X) +\nu(J_X(F)).$\checkmark

\end{enumerate}

If $\nu(\frac{\partial F}{\partial x})>\inf \{4\nu(x), \nu(t)+6\nu(y)\}$ \\
we must have  \fbox{$4\nu(x)= \nu(t)+6\nu(y)$}, and because $F\circ\gamma\equiv0$, we have that $x^5+txy^6=-y^7z- z^{15}.$\\

On the other hand $x^5+ty^6x=-4x^5+ x\frac{\partial F}{\partial x}$ and  $\nu(\frac{\partial F}{\partial x})>\inf \{4\nu(x), \nu(t)+6\nu(y)\}$ implies  that $\nu(x^5+ty^6x)=5\nu(x). $  Then $5\nu(x)\geq \inf\{7\nu(y)+\nu(z),15\nu(z)\}$\\
and
$5\nu(x)=\inf\{7\nu(y)+\nu(z),15\nu(z)\}$  if $\nu(y)\not=2\nu(z).$

\begin{enumerate}[(i)]

\item If  $\nu(y)>  2\nu(z)$ then $\nu(x)= 3\nu(z).$

Then $\nu(t)+\nu( \frac{\partial F}{\partial t})=\nu(t)+\nu(x)+6\nu(y)> \nu(t)+ 15\nu(z)>  \nu(z)+ \nu( \frac{\partial F}{\partial z})
\geq \nu(X) +\nu(J_X(F))$ \checkmark

\item If  $\nu(y)=  2\nu(z)$ 
\begin{enumerate}
\item  and  $\nu(x)= 3\nu(z)$, then from $ 4\nu(x)= \nu(t)+6\nu(y)$ we obtain $12\nu(z)= \nu(t)+12\nu(z)$ i.e. $\nu(t)=0$
absurd \checkmark
\item and $\nu(x)>3\nu(z)$, then $\nu(\frac{\partial F}{\partial z})=8\nu(z)$ and it follows that\\
$\nu(t)+\nu( \frac{\partial F}{\partial t})=\nu(t)+\nu(x)+6\nu(y)> \nu(t)+ 15\nu(z)>  \nu(z)+ \nu( \frac{\partial F}{\partial z})
\geq \nu(X) +\nu(J_X(F))$ \checkmark
\end{enumerate}

\item if  $\nu(y)<  2\nu(z)$, we have $5\nu(x)=7\nu(y)+\nu(z),$

and  the previous equation gives $\nu(x)+\nu(t)= \nu(y)+\nu(z)$.

\end{enumerate}
We can suppose  now \fbox{$\nu(x)+\nu(t)= \nu(y)+\nu(z)$} and  \fbox{$\nu(y)< 2\nu(z)$}.

We carry on with the last cases:

\begin{enumerate}[(I)]
\item
If $\nu(z)> \nu(y)$
we have $\nu(z^{14})>\nu(y^7)$ so $\nu(\frac{\partial F}{\partial z})=7\nu(y)$.\\
Then $\nu(t)+\nu( \frac{\partial F}{\partial t})=\nu(t)+\nu(x)+6\nu(y)=\nu(z)+7\nu(y)>\nu(y)+\nu(\frac{\partial F}{\partial z})\geq \nu(X) +\nu(J_X(F)).$\checkmark
\item
If $2\nu(z)>\nu(y)>\nu(z)$ then $\nu(\frac{\partial F}{\partial z})= 7\nu(y)>6\nu(y)+\nu(z) .$

 $y\frac{\partial F}{\partial y}= 6txy^6+7y^{7}z=6txy^6-7(x^5+txy^6+z^{ 15})=-7x^5 -txy^6 -7z^{ 15}$

since $15 \nu(z) >\nu(z)+7\nu(y)$ and $\nu(\frac{\partial F}{\partial x})>4\nu(x)$

so we must have $\nu(y\frac{\partial F}{\partial y})=\nu(txy^6)$ i.e. $\nu(\frac{\partial F}{\partial y})=\nu(txy^5).$

Then $\nu(t)+\nu( \frac{\partial F}{\partial t})=\nu(t)+\nu(x)+6\nu(y)> \nu(txy^5)+\nu(z)\geq \nu(\frac{\partial F}{\partial y})+\nu(z)\geq \nu(X) +\nu(J_X(F)).$\checkmark

\end{enumerate}

{\bf Résumé: a germ of  arc along which Whitney condition $(b)$  is not satisfied must fulfil the following conditions:
\begin{itemize}
\item $\nu(x)>\nu(y)=\nu(z)> \nu(t)$
\item
$\nu(x)+\nu(t)=\nu(z)+\nu(y)$
\item
$4\nu(x)=\nu(t)+6\nu(y)$

\end{itemize}

}

Finally the set of germs of analytic arcs along which Whitney condition $(b)$ is not satisfied is contained in the set

$$\mathcal{ A}:=\left\{\right.\gamma(s)=(x(s),y(s),z(s),t(s)) :[0,1]\to \mathbb{C}^n\times \mathbb{C}\,| $$
$$\begin{cases}
 x(s)=a_1s^{\alpha_1}+\ldots\\
y(s)= a_2s^{\alpha_2}+\ldots\\
z(s)=a_3s^{\alpha_3}+\ldots\\
t(s)=a_4s^{\alpha_4}+\ldots
\end{cases}$$
$$ 5\alpha_1=8\alpha, \alpha_2=\alpha_3=\alpha, 5\alpha_4=2\alpha, \alpha\equiv 0[5], a_i\in \mathbb{C}^*
\text{ satisfying some conditions} \left\}\right..$$

It remains to characterize the subset of arcs along which the $\delta $ condition is not satisfied.

Let $\gamma\in \mathcal{A}$,  we may suppose $a_1=1$, and write $a_2=a, a_3=b$ and $a_4=c$.

$F\circ\gamma(s)=(s^{ 8\alpha}+...) +( c.a^6 s^{ 8\alpha}+ \ldots) + (a^7 b s^{ 8\alpha}
+\ldots) + (b^{ 15}s^{ 15\alpha }+\ldots )\equiv0$

so then $ s^{8\alpha}(1+c.a^6+ a^7.b+s(\ldots+ b^{15}s^{  7\alpha-1}+ \ldots))\equiv0$,

and we must have 

\fbox {$1+c.a^6+ a^7.b=0$}.

Then along $\gamma(s)$  near $s=0$ we have,

$$\begin{cases}\frac{\partial F}{\partial x}=5x^4+ty^6= s^{ \frac{32}{5}\alpha}(5+ca^6)+ \ldots
 \\
 \frac{\partial F}{\partial y}=6txy^5+7y^{6}z =a^5(6c+7ab)s^{ 7\alpha}+ \ldots\\
  \frac{\partial F}{\partial z}=y^7+15z^{14}=a^7s^{ 7\alpha} +\ldots +14b^{ 14}  s^{  14\alpha}+\ldots.

\end{cases}$$
But, $\nu( \frac{\partial F}{\partial x})> 7\alpha$ imposes the condition 
\fbox{$5+ca^6=0$}.

It follows that  $c=-\frac{5}{a^6}$, $b=\frac{4}{a^7}$ and $a^5(6c+7ab)=-\frac{2}{a}.$

The limit of orthogonal secant vectors $\frac{(x,y,z)}{\|(x,y,z)\|}$ is  $(0: a :  b)=(0: a^8 :  4) $,  and the limit of normal vectors $ \frac{ grad_{x} F(x,y,z,t)}{\| grad_{x} F(x,y,z,t)\|}$ is 
 $(0 : a^5(6c+7ab) : a^7  )=(0 : -\frac{2}{a} : a^7  ) =(0 : -2 : a^8 ). $
 It follows that $(\delta^{\pi})$  is not satisfied along $\gamma$ if and only if
$a^{16}=-8$.  Choosing $a$ to be one of these 16 complex 
numbers,  we have the desired conclusion, namely that $(\delta^{\pi})$ fails precisely on those curves $\gamma(s)= (x(s), y(s), z(s), t(s))$ whose initial terms are $ (s^8, a s^5, 4 a^{-7} s^5,  - 5 a^{-6} s^2)$.

By Lemma 1, and $(a)$-regularity, these are precisely the curves on which $(\delta)$ fails, that is to say we have identified all of the curves on which weak Whitney regularity fails to hold.

\section{	Other  Brian\c con and Speder examples}

   We perform similar calculations for the infinite family of examples, also due to Brian\c con and Sperder  \cite{10} : $F(x,y,z,t) = x^3+txy^\alpha+ y^\beta z+z^{ 3\alpha}$, 
where \fbox{$\alpha\geq 3$} and \fbox{$3\alpha=2\beta+1$}.\\
The functions $f_t(x,y,z) = F_t(x,y,z)$ are  quasihomogenous 
of type $(\alpha ,2,1; 3\alpha )$ with isolated singularity at the origin, for each $t$, 
and  so each $\mu_t=(3\alpha-1)(3\alpha-2)=2\beta(2\beta-1)$. Thus $f_t$ is a $\mu$-constant family . 

We are again hunting for analytic arc germs where  condition $(\delta)$ fails.

$$\left\{\begin{array}{lll}
 \frac{\partial F}{\partial x}=3x^2+ty^\alpha
 \\
 \frac{\partial F}{\partial y}=\alpha	txy^{ \alpha-1}+\beta	y^{\beta-1}z \\
  \frac{\partial F}{\partial z}=y^\beta+3\alpha	z^{3\alpha-1}=y^\beta+3\alpha z^{2\beta}.
  \end{array}\right.$$

Let $\gamma:([0,1],0)\to ( \mathbb{C}^n\times \mathbb{C},0)$, be a germ of an analytic arc and $\nu$ the valuation along $\gamma.$

Let  $X=(x,y,z)$ and $J_XF=( \frac{\partial F}{\partial x}, \frac{\partial F}{\partial y}, \frac{\partial F}{\partial z})$, then write $\nu(X):=\inf\{\nu(x), \nu(y),\nu(z)\}$ and $\nu(J_X(F)):=\inf\{\nu(\frac{\partial F}{\partial x}),\nu(\frac{\partial F}{\partial y}), \nu(\frac{\partial F}{\partial z})\}.$

We  begin by  determining  along  which curves condition $(b^{\pi})$ holds (in this case it is equivalent to $(b)$, because $(a)$ holds (this is a consequence of having constant Milnor number \cite{18}, and may also be checked by direct calculation). 

Because of  lemma \ref{simple} and the $\mu$-constant condition,  if $\nu(t)\geq \nu(X)$ then $(b^{\pi})$ holds, because 
 ($\nu(t)+\nu( \frac{\partial F}{\partial t})=\nu(t)+\nu(x)+\alpha\nu(y)>\nu(x) +\nu(J_x(F))$,  which means that Whitney's condition $(b)$ holds when  $\nu(t)\geq\nu(X)$.\checkmark

We suppose from now on that  \fbox{$\nu(t)<\nu(X)$}.

If $\nu(\frac{\partial F}{\partial x})=\inf \{2\nu(x), \nu(t)+\alpha\nu(y)\}$ we have:
\begin{enumerate}
\item
if $2\nu(x)\geq \nu(t)+\alpha\nu(y),$ we must have $\nu(x)>\nu(y)$  then  $\nu(t)+\nu( \frac{\partial F}{\partial t})=\nu(t)+\nu(x)+\alpha\nu(y)>\nu(y)+ (\nu(t)+\alpha\nu(y))
\geq \nu(X) +\nu(J_x(F)).$\checkmark
\item
if $2\nu(x)<\nu(t)+\alpha\nu(y)$ then  $\nu(t)+\nu( \frac{\partial F}{\partial t})=\nu(t)+\nu(x)+\alpha\nu(y)>\nu(x)+ (3\nu(x))
\geq \nu(X) +\nu(J_x(F)).$\checkmark

\end{enumerate}

If $\nu(\frac{\partial F}{\partial x})>\inf \{2\nu(x), \nu(t)+\alpha\nu(y)\}$ \\
we must have  \fbox{$2\nu(x)= \nu(t)+\alpha\nu(y)$}, and because $F\circ\gamma\equiv0$, i.e. $x^3+txy^\alpha=- y^\beta z-z^{ 3\alpha}.$\\

On the other hand $x^3+txy^\alpha=-2x^3+ x\frac{\partial F}{\partial x}$ and  $\nu(\frac{\partial F}{\partial x})>\inf \{2\nu(x), \nu(t)+\alpha\nu(y)\}$ implies  $\nu(x^3+txy^\alpha)=3\nu(x). $  Then $3\nu(x)\geq \inf\{\beta\nu(y)+\nu(z),3\alpha\nu(z)\}$\\
and
$3\nu(x)=\inf\{\beta\nu(y)+\nu(z),3\alpha\nu(z)\}$  if $\nu(y)\not=2\nu(z).$

\begin{enumerate}[(i)]

\item If  $\nu(y)>  2\nu(z)$ then $\nu(x)= \alpha\nu(z).$

Then $\nu(t)+\nu( \frac{\partial F}{\partial t})=\nu(t)+\nu(x)+\alpha\nu(y)> \nu(t)+ 3\alpha\nu(z)>  \nu(z)+ \nu( \frac{\partial F}{\partial z})
\geq \nu(X) +\nu(J_X(F))$ \checkmark

\item If  $\nu(y)=  2\nu(z)$ 
\begin{enumerate}
\item  and  $\nu(x)= \alpha\nu(z)$, then from $ 2\nu(x)= \nu(t)+\alpha\nu(y)$ we obtain $2\alpha\nu(z)= \nu(t)+2\alpha\nu(z)$ i.e. $\nu(t)=0$ absurd \checkmark
\item and $\nu(x)>\alpha\nu(z)$ then $\nu(\frac{\partial F}{\partial z})=2\beta\nu(z)$ it follows that\\
$\nu(t)+\nu( \frac{\partial F}{\partial t})=\nu(t)+\nu(x)+\alpha\nu(y)> \nu(t)+ 3\alpha\nu(z)>  \nu(z)+ \nu( \frac{\partial F}{\partial z})
\geq \nu(X) +\nu(J_X(F))$ \checkmark
\end{enumerate}

\item If  $\nu(y)<  2\nu(z)$, we have $3\nu(x)=\beta\nu(y)+\nu(z).$

and  the previous equation gives $\nu(x)+\nu(t)= (\beta-\alpha)\nu(y)+\nu(z)$

\end{enumerate}
We can suppose  now \fbox{$\nu(x)+\nu(t)= (\beta-\alpha)\nu(y)+\nu(z)$} and  \fbox{$\nu(y)< 2\nu(z)$}

We carry on with the last cases:

\begin{enumerate}[(I)]
\item
If $\nu(z)> \nu(y)$
we have $\nu(z^{3\alpha-1})=\nu(z^{2\beta})>\nu(y^\beta)$ so $\nu(\frac{\partial F}{\partial z})=\beta\nu(y)$.\\
Then $\nu(t)+\nu( \frac{\partial F}{\partial t})=\nu(t)+\nu(x)+\alpha\nu(y)=\nu(z)+\beta\nu(y)>\nu(y)+\nu(\frac{\partial F}{\partial z})\geq \nu(X) +\nu(J_X(F)).$\checkmark
\item
If $2\nu(z)>\nu(y)>\nu(z)$ then $\nu(\frac{\partial F}{\partial z})= \beta\nu(y)>(\beta-1)\nu(y)+\nu(z) .$

 $y\frac{\partial F}{\partial y}= \alpha txy^\alpha+\beta y^{\beta}z= \alpha txy^\alpha-\beta(x^3+txy^\alpha+z^{ 3\alpha})=-\beta	x^3-(\beta-\alpha) txy^\alpha-\beta z^{ 3\alpha}$

since $3\alpha=2\beta+1$, we get   $3\alpha \nu(z) >\nu(z)+\beta\nu(y)$ and  from $\nu(\frac{\partial F}{\partial x})>2\nu(x)$

we must have $\nu(y\frac{\partial F}{\partial y})=\nu(txy^\alpha)$ i.e. $\nu(\frac{\partial F}{\partial y})=\nu(txy^{ \alpha-1}).$

Then $\nu(t)+\nu( \frac{\partial F}{\partial t})=\nu(t)+\nu(x)+\alpha\nu(y)= \nu(txy^{ \alpha-1})+\nu(y)>\nu(txy^{ \alpha-1})+\nu(z)\geq \nu(\frac{\partial F}{\partial y})+\nu(z)\geq \nu(X) +\nu(J_X(F)).$\checkmark

\end{enumerate}

{\bf Résumé: a germ of  arc along which Whitney condition $(b)$ is not satisfied must fulfil the following conditions:
\begin{itemize}
\item $\nu(x)>\nu(y)=\nu(z)> \nu(t)$
\item
$\nu(x)+\nu(t)=\nu(z)+(\beta-\alpha)\nu(y)$
\item
$2\nu(x)=\nu(t)+\alpha\nu(y)$

\end{itemize}

}

Finally   the set of germs of analytic arcs along which Whitney condition $(b)$ is not satisfied is contained in the set

$$\mathcal{ A}:=\left\{\right.\gamma(s)=(x(s),y(s),z(s),t(s)) :[0,1]\to \mathbb{C}^n\times \mathbb{C}\,| $$
$$\begin{cases}
 x(s)=a_1s^{\alpha_1} +\ldots\\
y(s)= a_2s^{\alpha_2} +\ldots\\
z(s)=a_3s^{\alpha_3} +\ldots\\
t(s)=a_4s^{\alpha_4} +\ldots
\end{cases}$$
$$ 3\alpha_1=(\beta+1)m, \alpha_2=\alpha_3=m, 3\alpha_4=m, \alpha\equiv 0[3], a_i\in \mathbb{C}^*
\text{ satisfying some conditions} \left\}\right..$$

It remains to characterize the subset of arcs along which the $(\delta )$ condition is not satisfied.

Let $\gamma\in \mathcal{A}$,  we may suppose $a_1=1$, and write $a_2=a, a_3=b$ and $a_4=c$

$F\circ\gamma(s)=(s^{ (\beta+1)m}+...) +( c.a^\alpha s^{ (\beta+1)m}+ \ldots) + (a^\beta b s^{ (\beta+1)m}
+\ldots) + (b^{ 3\alpha }s^{ 3\alpha m }+\ldots )\equiv0$

then $ s^{(\beta+1)m}(1+a^\alpha.c+ a^\beta.b+s(\ldots+ b^{3\alpha}s^{  \beta m}+ \ldots))\equiv0$

and we must have 

\fbox {$1+c.a^\alpha+ b.a^\beta=0$}.

Then along $\gamma(s)$  near $s=0$ we have,

$$\begin{cases}\frac{\partial F}{\partial x}=3x^2+ty^\alpha= s^{ \frac{2(\beta+1) m}{3}\alpha}(3+ca^\alpha)+ \ldots
 \\
 \frac{\partial F}{\partial y}=\alpha	txy^{ \alpha-1}+\beta	y^{\beta-1}z =(\alpha	ca^{ \alpha-1}+\beta b.a^{ \beta-1})s^{ \beta m}+ \ldots\\
  \frac{\partial F}{\partial z}=y^\beta+3\alpha	z^{3\alpha-1}=y^\beta+3\alpha z^{2\beta}=a^\beta	 s^{ \beta m} +\ldots +(2\beta)	b^{ 2\beta)}  s^{  (2\beta)m}+\ldots.

\end{cases}$$
But, the condition $\nu( \frac{\partial F}{\partial x})> \beta m$ implies
\fbox{$3+ca^\alpha=0$}

and it follows that  $c=-\frac{3}{a^\alpha}$, $b=\frac{2}{a^\beta}$ and $\alpha c.a^{ \alpha-1}+\beta b.a^{ \beta-1}=-\frac{1}{a}.$

The limit of orthogonal secant vectors $\frac{(x,y,z)}{\|(x,y,z)\|}$ is  $(0: a :  b)=(0: a :  \frac{2}{a^\beta}) $,  and the limit of normal vectors $ \frac{ grad_{x} F(x,y,z,t)}{\| grad_{x} F(x,y,z,t)\|}$ is 
 $(0 : \alpha c.a^{ \alpha-1}+\beta b.a^{ \beta-1} : a^\beta  )=(0 : -\frac{1}{a} : a^\beta  ). $
 It follows that $(\delta)$  is not satisfied along $\gamma$ if and only if
$a^{2\beta+2}=-2$.  Choosing $\alpha$ to be one of these $2\beta+2=3\alpha+1$ complex 
numbers,  we have the desired conclusion, i.e. that $(\delta)$ fails.

\section{Other examples.}

A Milnor number constant  family, $F_t (x,y,z) = z^{12} + zy^3x + ty^2x^3 + x^6 + y^5$, with $\mu = 166$, which is also not Whitney regular over the $t$-axis, was studied by E. Artal Bartolo, J. Fernandez de Bobadilla, I. Luengo and  A. Melle-Hernandez in a recent paper \cite{2}.
A series of Milnor number constant but non Whitney regular families, depending on a parameter $\ell$, was given by Abderrahmane \cite{1} as follows: $F^{\ell}_t (x,y,z) = x^{13} + y^{20} + z x^6 y^5 + t x^6 y^8 + t^2 x^{10} y^3 + z^{\ell},$ for integers $\ell \geq 7$. Here $\mu = 153 \ell + 32$, while $\mu^2 (F_0) = 260$ and $\mu^2 (F_t) = 189$, according to Abderrahmane.
We do not yet know whether weak Whitney regularity holds or fails for these examples.

\end{document}